\documentclass{amsart}
\usepackage{amscd,amssymb,amsopn,amsmath,amsthm,graphics,amsfonts,accents,enumerate,verbatim,calc}
\usepackage[dvips]{graphicx}
\usepackage[colorlinks=true,linkcolor=red,citecolor=blue]{hyperref}
\usepackage[all]{xy}
\usepackage{cite}
\usepackage{tikz}
\usepackage{tikz-cd}
\usepackage{mathrsfs}

\usepackage[latin1,utf8]{inputenc}

\calclayout

\newcommand{\rt}{\rightarrow}
\newcommand{\lrt}{\longrightarrow}

\newcommand{\va}{\vartheta}
\newcommand{\st}{\stackrel}

\newcommand{\la}{\lambda}
\newcommand{\La}{\Lambda}
\newcommand{\Ga}{\Gamma}
\newcommand{\Om}{\Omega}
\newcommand{\lan}{\langle}
\newcommand{\ran}{\rangle}

\newcommand{\Up}{\Upsilon}

\newcommand{\Z}{\mathbb{Z} }

\newcommand{\A}{\mathcal{A}}
\newcommand{\SB}{\mathscr{B}}

\newcommand{\SK}{\mathscr{K}}
\newcommand{\M}{\mathcal{M}}
\newcommand{\SP}{\mathcal{P}}
\newcommand{\SR}{\mathscr{R}}

\newcommand{\SX}{\mathscr{X}}
\newcommand{\SY}{\mathscr{Y}}

\newcommand{\CA}{\mathcal{A} }

\newcommand{\CF}{\mathcal{F} }

\newcommand{\CK}{\mathcal{K} }

\newcommand{\CM}{\mathcal{M} }

\newcommand{\CS}{\mathcal{S} }

\newcommand{\CX}{\mathcal{X} }

\newcommand{\CW}{\mathcal{W}}
\newcommand{\CV}{\mathcal{V}}
\newcommand{\CU}{\mathcal{U}}

\newcommand{\CH}{\mathcal{H}}

\newcommand{\ora}{\overrightarrow}
\newcommand{\ola}{\overleftarrow}

\newcommand{\Mod}{{\rm{Mod\mbox{-}}}}

\newcommand{\mmod}{{\rm{{mod\mbox{-}}}}}
\newcommand{\mmodd}{{\rm{mod}}_0\mbox{-}}
\newcommand{\unmod}{\underline{{\rm mod}}\mbox{-}}

\newcommand{\prj}{{\rm{prj}\mbox{-}}}

\newcommand{\inj}{{\rm{inj}\mbox{-}}}

\newcommand{\op}{{\rm{op}}}

\newcommand{\add}{{\rm{add}\mbox{-}}}

\newcommand{\gldim}{{\rm{gl.dim}}}

\newcommand{\domdim}{{\rm{dom.dim}}}

\newcommand{\Coker}{{\rm{Coker}}}
\newcommand{\Ker}{{\rm{Ker}}}

\newcommand{\Hom}{{\rm{Hom}}}
\newcommand{\Ext}{{\rm{Ext}}}
\newcommand{\End}{{\rm{End}}}

\theoremstyle{plain}
\newtheorem{theorem}{Theorem}[section]
\newtheorem{corollary}[theorem]{Corollary}
\newtheorem{lemma}[theorem]{Lemma}

\newtheorem{proposition}[theorem]{Proposition}

\newtheorem{notation}[theorem]{Notation}

\theoremstyle{definition}

\newtheorem{remark}[theorem]{Remark}

\newtheorem{s}[theorem]{}

\theoremstyle{plain}
\newtheorem{stheorem}{Theorem}[subsection]
\newtheorem{scorollary}[stheorem]{Corollary}

\newtheorem{sproposition}[stheorem]{Proposition}

\theoremstyle{definition}

\newtheorem{sremark}[stheorem]{Remark}

\newtheorem{sss}[stheorem]{}

\numberwithin{equation}{section}

\begin{document}

\title[Monomorphism category of $n$-cluster tilting subcategories]{On the Monomorphism Category of $n$-Cluster Tilting Subcategories}

\author[J. Asadollahi, R. Hafezi and S. Sadeghi]{Javad Asadollahi, Rasool Hafezi and Somayeh Sadeghi}

\address{Department of Pure Mathematics, Faculty of Mathematics and Statistics, University of Isfahan, P.O.Box: 81746-73441, Isfahan, Iran}
\email{asadollahi@ipm.ir, asadollahi@sci.ui.ac.ir}
\email{so.sadeghi@sci.ui.ac.ir }

\address{School of Mathematics, Institute for Research in Fundamental Sciences (IPM), P.O.Box: 19395-5746, Tehran, Iran}
\email{hafezi@ipm.ir}




\begin{abstract}
Let $\mathcal{M}$ be an $n$-cluster tilting subcategory of ${\rm mod}\mbox{-}\Lambda$, where $\Lambda$  is an artin algebra. Let $\mathcal{S}(\mathcal{M})$ denotes the full subcategory of $\mathcal{S}(\Lambda)$, the submodule category of $\Lambda$, consisting of all monomorphisms in $\mathcal{M}$. We construct two functors from $\mathcal{S}(\mathcal{M})$ to ${\rm mod}\mbox{-}\underline{\mathcal{M}}$, the category of finitely presented (coherent) additive contravariant functors on the stable category of $\mathcal{M}$. We show that these functors are full, dense and objective. So they induce equivalences from the quotient categories of the submodule category of $\mathcal{M}$ modulo their respective kernels.  Moreover, they are related by a syzygy functor on the stable category of ${\rm mod}\mbox{-}\underline{\mathcal{M}}$. These functors can be considered as a higher version of the two functors studied by Ringel and Zhang \cite{RZ} in the case $\Lambda=k[x]/{\langle x^n \rangle}$ and generalized later by Eir\'{i}ksson \cite{E} to self-injective artin algebras. Several applications will be provided.
\end{abstract}

\maketitle


\section{Introduction}
Let $R$ be a commutative artinian ring. Let $\La$ be an artin $R$-algebra of finite representation type. Let $\Ga:=\End_{\La}(E)^{\op}$ be the Auslander algebra of $\La$, where $E$ is an additive generator of $\mmod\La$, the category of finitely presented $\La$ modules. Moreover let $\mmod T_2(\La)$ denote the upper triangular matrix algebra. It is known that $\mmod T_2(\La)$ is equivalent to $\CH(\La)$, the morphism category of $\La$. So the objects of $\mmod T_2(\La)$ can be considered as morphisms in $\mmod\La$. Auslander \cite{Au1} studied a functor from $\mmod T_2(\La)$ to $\mmod \Ga$ by sending an object $f$ of $\mmod T_2(\La)$ to the cokernel of the induced map $\Hom_{\La}(E, f)$. This functor is usually denoted by $\alpha$. It then has been studied further by Auslander and Reiten \cite{AR1} and \cite{AR2}.

Let $n$ be a fixed positive integer. Let $\CS(n)$ be the submodule category of $\mmod\La_n$, where $\La_n=k[x]/{\lan x^n \ran}$ and $k$ is a field. Having the functor $\alpha$ as an `essential tool', Ringel and Zhang \cite{RZ} introduced and studied two functors $F$ and $G$ from $\CS(n)$ to $\mmod \Pi_{n-1}$, the preprojective algebra of type $\mathbb{A}_{n-1}$. They showed that $\mmod \Pi_{n-1}$ is isomorphic to $\underline{\Ga}$, the stable Auslander algebra of $\La$. Hence $F$ and $G$ are functors from $\CS(n)$ to $\underline{\Ga}$. They proved that these functors are full, dense and objective and hence they induce equivalences between the quotients of $\CS(n)$ by the ideals of the kernel objects of $F$ and $G$ and $\underline{\Ga}$. This, in particular, introduces quotients of $\CS(n)$ that are abelian categories with enough projective objects. They also provided a comparison of $F$ and $G$ and showed that they differ only by the syzygy functor on the stable module category $\underline{\Ga}$. Later, Eir\'{i}ksson \cite{E} studied these functors in a more general setting of representation finite self-injective artin algebras, thus he studied two functors from the submodule category of a representation-finite self-injective algebra $\La$ to the module category of the stable Auslander algebra of $\La$
\cite[\S 4]{E}.

Let us be a little bite more explicit. Consider the compositions
\[\xymatrix{\CS(\La) \ar@<0.5ex>[r]^{ \eta \ \ \ \ }  \ar@<-0.5ex>[r]_{\epsilon \ \ \ \ } & \mmod T_2(\La) \ar[r]^{\alpha} & \mmod \Ga \ar[r]^q  & \mmod\underline{\Ga}}\]
where $\eta$ is the inclusion of $\CS(\La)$ in $\mmod T_2(\La)$ and $\epsilon$ maps a morphism $f$ in $\CS(\La)$ to $\Coker(f)$ in $\mmod T_2(\La)$. $\alpha$ is the functor introduced by Auslander, we just recalled above, and for details on $q$ see \cite[\S 3]{E}. The functors $F$ and $G$ are given by $F=q\alpha\eta$ and $G=q\alpha\epsilon$. Note that the functor $F$ also was studied by Li and Zhang \cite{LZ}. Eir\'{i}ksson \cite[Theorem 1]{E} proved that the functor $F$ induces an equivalence of categories $\CS(\La)/{\CU} \lrt \mmod\underline{\Ga}$, where $\CU$ is the additive subcategory of $\CS(\La)$ generated by all objects of the form $M \rt M$ and $M \rt I$, where $M \in \mmod\La$ and $I$ is an injective-projective module. Moreover, $G$ induces an equivalence of categories $\CS(\La)/{\CV} \lrt \mmod\underline{\Ga}$, where $\CV$ is the smallest additive subcategory of $\CS(\La)$ generated by all objects of the form $M \rt M$ and $0 \rt M$, where $M \in \mmod\La$.

Roughly speaking, our aim in this paper is to introduce the above functors to the higher homological algebra setting. This theory is born while Iyama developed a higher version of Auslander's correspondence and Auslander-Reiten theory for artin algebras and related rings, see e.g. \cite{I1}, \cite{I2}. The notions of $n$-cluster tilting modules and $n$-cluster tilting subcategories are fundamental in the Iyama's theory, see for instance \cite{I3}.

Although $n$-cluster tilting subcategories of abelian categories are not abelian, Jasso \cite{J} proved that they have a very nice structure, known as $n$-abelian structure. In this new structure, special sequences of length $n+2$ play the role of short exact sequences in abelian categories. Higher homological algebra is currently a very active area of research. Its importance stems from the many connections and applications cluster tilting theory has in many research areas: Algebraic and Quantum groups (total positivity and canonical bases), Representation Theory (in particular representations of quivers), Geometry (Poisson Geometry, Teichm\"{u}ller spaces, integrable systems), Combinatorics (Stasheff assosiahedra), Algebraic Geometry (Bridgeland's stability conditions, Calabi-Yau algebras, Donaldson-Thomas invariants) and Non-Commutative Geometry (non-commutative crepant resolutions).  For basics of the theory see Subsection \ref{HigherHomAlgebra}, below.

Let $\M$ be an $n$-cluster tilting subcategory of $\mmod\La$. Let $\CS(\M)$ denote the subcategory of $\CS(\La)$ consisting of all monomorphisms in $\M$. We introduce and study two functors $\Phi$ and $\Psi$ from $\CS(\M)$ to $\mmod\underline{\M}$, where $\mmod\underline{\M}$ is the category of additive contravariant finitely presented functors from $\underline{\M}$ to $\CA b$, the category of abelian groups. We show that these two functors, provide equivalences between the quotient categories of $\CS(\M)$ and $\mmod\underline{\M}$. We also compare these two functors and show that they differ by the $n$-th syzygy functor, provided $\M$ is an $n\Z$-cluster tilting subcategory.

The paper is structured as follows. After the introduction, in Section \ref{Pre}, we provide some backgrounds that we need throughout the paper. Section \ref{Phi} is devoted to introduce and study the functor $\Phi$ and in Section \ref{Psi} we will investigate the functor $\Psi$. Section \ref{Comp} contains a comparison of these two functors, when $\La$ is a self-injective artin algebra and  $\M$ is an $n\Z$-cluster tilting subcategory of $\mmod\La$. Section \ref{Sec-Dual} is devoted to a list of duals of the results in Sections \ref{Phi} and \ref{Psi}. Since proofs are similar, we just list the statements without proof. In the last section we provide some applications. In particular, we
present a duality from $\mmod\underline{\M}$ to $\overline{\M}\mbox{-}{\rm mod}$ (Corollary \ref{Corollary-Auslander}), that could be considered as a higher version of the Auslander's result \cite{Au3} showing the existence of a duality between $\mmod\underline{\A}$ and $\overline{\A}\mbox{-}{\rm mod}$, where $\CA$ is an abelian category. We use this duality to prove a higher version of Hilton-Rees Theorem for $n$-cluster tilting subcategories (Theorem \ref{Th-HiltonRees}), and a higher version of Auslander's direct summand conjecture (Theorem \ref{Th-DirectSummand}). Moreover, we apply our results to reprove the existence of $n$-Auslander-Reiten translation $\tau_n=\tau\Omega^{n-1}_{\La}$ in $n$-cluster tilting subcategories  (Theorem \ref{Th-nAuslanderReiten}). Finally, we establish an equivalence between the stable categories of the functors of projective dimension at most one (Proposition \ref{stableequiv}).

\section{Preliminaries}\label{Pre}
Let $\A$ be an abelian category and $\M$ be a full additive subcategory of $\A$. For an object $A \in \A$, let $\A( - , A) \vert_{\M}$ denote the functor $\A( - , A)$ restricted to $\M$.  A right $\M$-approximation of $A$ is a morphism $\pi : M \rt A$ with $M \in \M$ such that $\A( - , M)\vert_{\M} \lrt \A( - , A)\vert_{\M} \lrt 0$ is exact. $\M$ is called a contravariantly finite subcategory of $\A$ if every object of $\A$ admits a right $\M$-approximation. Dually, the notion of left $\M$-approximations and covariantly finite subcategories are defined. $\M$ is called functorially finite subcategory of $\A$, if it is both contravariantly and covariantly finite.

A subcategory $\M$ of $\A$ is called a generating subcategory if for every object $A \in \A$, there exists an epimorphism $M \lrt A$ with $M \in \M$.
Cogenerating subcategories are defined dually. $\M$ is called a generating-cogenerating subcategory of $\A$ if it is both a generating and a cogenerating subcategory.\\

\s{\sc Higher homological algebra.}\label{HigherHomAlgebra}
The concept of $n$-abelian categories is formalized and studied in \cite{J} as a generalisation of the notion of abelian categories. Let us recall the basics. Let $\M$ be an additive category. Let $u^0 : M^0 \lrt M^{1}$  be a morphism in $\M$. An $n$-cokernel of $u^0$  is a sequence
\[M^1 \st{ u^1 }{\lrt } M^2\lrt  \cdots \lrt M^{n}\st{ u^{n} }{ \lrt } M^{n+1}\]
of morphisms in $\M$ such that for every $M \in \M$, the induced sequence
\[0 \lrt \M(M^{n+1}, M) \st{u^n _*}{\lrt} \cdots \st{u^{ 1 }_*}{\lrt} \M(M^1, M) \st{u^0_*}{\lrt} \M(M^{0}, M)\]
of abelian groups is exact. $n$-cokernel of $u^0$ denotes by $(u^1, u^2, \cdots, u^{n})$. The notion of $n$-kernel of a morphism $u^n:M^n \lrt M^{n+1}$ is defined similarly, or rather dually.

A sequence $M^0 \st{u^0 }{\lrt } M^1 \lrt  \cdots \lrt M^n \st{u^n}{\lrt} M^{n+1}$ of objects and morphisms in $\M$ is called $n$-exact \cite[Definitions 2.2, 2.4]{J} if $(u^0, u^1, \cdots, u^{n-1})$ is an $n$-kernel of $u^n$ and $(u^1, u^2, \cdots, u^{n})$ is an $n$-cokernel of $u^0$. An $n$-exact sequence like the above one, will be denoted by
\[0 \lrt M^0 \st{ u^0 }{ \lrt }M^1 \lrt  \cdots \lrt M^n \st{ u^n }{\lrt } M^{n+1} \lrt 0.\]

The additive category $\M$ is called $n$-abelian \cite[Definition 3.1]{J} if it is idempotent complete, each morphism in $\M$ admits an $n$-cokernel and an $n$-kernel and every monomorphism $u^0: M^0 {\lrt } M^1$, respectively every epimorphism $u^n: M^n { \lrt } M^{n+1}$, can be completed to an $n$-exact sequence
\[0 \lrt M^0\st{u^0}\lrt M^1\lrt \cdots \lrt M^n \st{u^n} \lrt M^{n+1} \lrt 0.\]

Let $\A$ be an abelian category. An additive subcategory $\M$ of $\A$ is called an $n$-cluster tilting subcategory \cite[Definition 3.14]{J} if it is a functorially finite and generating-cogenerating subcategory of $\A$ such that $\M^{\perp_n}=\M= {}^{\perp_n}\M$, where
\[\M^{\perp_n}:= \{A \in \A \mid \Ext^i_{\A}(\M, A)=0 \ \text{ for all} \ 0 < i<  n \}, \]
\[{}^{\perp_n}\M:= \{A \in \A \mid \Ext^i_{\A}(A, \M)=0 \ \text{ for all} \ 0 < i <  n \}.\]

It is known that every  $n$-cluster tilting subcategory of an abelian category $\A$ has a structure as an $n$-abelian category \cite[Theorem 3.16]{J}. On the other hand, every small $n$-abelian category $\M$ is equivalent to an $n$-cluster tilting subcategory of an abelian category $\A$, see \cite[Theorem 4.3]{EN} and \cite[Theorem 7.3]{Kv}.

\s{\sc Morphism category.}
Let $\A$ be an abelian category. The morphism category of $\A$, denoted by $\CH(\A)$, is a category whose objects are morphisms in $\A$. Its morphisms are given by commutative diagrams and composition is defined naturally \cite{RS}. Let $f: A \lrt B$ be an object of $\CH(\A)$. Then $A$ (resp. $B$) is called the source (resp. target) of $f$, denoted by $s(f)$ (resp. $t(f)$). It is known that $\CH(\A)$ is an abelian category. A sequence
\[\xymatrix{0 \ar[r] & f' \ar[r]^{\alpha}_\beta & f \ar[r]^{\delta}_{\gamma}  & f'' \ar[r] & 0}\]
of morphisms in $\CH(\A)$ is exact if and only if the induced sequences of sources and targets are exact in $\A$. \\
It has two important full subcategories, i.e. the monomorphism and the epimorphism categories, denoted respectively by $\CS(\A)$ and $\CF(\A)$. As it is expected from their names, the objects of $\CS(\A)$, resp. the objects of $\CF(\A)$, are monomorphisms, resp. epimorphisms, in $\CH(\A)$.
They both are closed under extensions and summands, $\CS(\A)$ is closed under taking kernels and $\CF(\A)$ is closed under taking cokernels.
By \cite[Appendix A]{Ke} a full extension closed subcategory of an abelian category is a Quillen exact category. Hence $\CS(\A)$ and $\CF(\A)$ are both exact categories. The conflations of $\CS(\A)$ (resp. $\CF(\A)$) are exact sequences
\[\xymatrix{0 \ar[r] & f' \ar[r]^{\alpha}_\beta & f \ar[r]^{\delta}_{\gamma}  & f'' \ar[r] & 0}\]
in $\CH(\A)$ with terms in $\CS(\A)$ (resp. $\CF(\A)$).

\s{\sc Functor category.} \label{Functor} Let $\SX$ be a skeletally small additive category. By definition, a (right) $\SX$-module is a contravariant additive functor $F:\SX \rt \CA b$, where $\CA b$ denotes the category of abelian groups. The $\SX$-modules and natural transformations between them form an abelian category denoted by $\Mod\SX$. An $\SX$-module $F$ is called finitely presented if there exists an exact sequence
\[\SX( - ,X) \rt \SX( - ,X') \rt F \rt 0,\]
with $X$ and $X'$ in $\SX$. Finitely presented $\SX$-modules form a full subcategory of $\Mod\SX$, denoted by $\mmod\SX$. It is proved by Auslander
\cite[Chapter III, \S 2]{Au2} that $\mmod\SX$ is an abelian category if and only if $\SX$ admits weak kernels. This happens, for example, when $\SX$ is a contravariantly finite subcategory of an abelian category $\CA$. We let $\mmod\underline{\SX}$ be the category of all functors $F \in \mmod\SX$ that vanishes on projective modules. So $\mmod\underline{\SX}$ can be considered as a subcategory of $\mmod\SX$.

The category of covariant additive functors $F:\SX \rt \CA b$ is denoted by $\SX\mbox{-}{\rm Mod}$, called the category of left $\SX$-modules. Moreover,  $\SX\mbox{-}{\rm mod}$ denotes its subcategory consisting of finitely presented left $\SX$-modules.

\s{\sc Objective functors.} \label{Objective} Here we recall some facts on objective functors. For a good reference see \cite[Appendix]{RZ}. Let $F : \SX \lrt \SY$ be an additive functor between additive categories. $F$ is called an objective functor if any morphism $f$ in $\SX$ with $F(f) = 0$ factors through an object $K$ with $F(K) = 0$. $K$ is then called a kernel object of $F$. We say that the kernel of an objective functor $F$ is generated by $\SK$ if $\add\SK$ is the class of all kernel objects of $F$.

Let $F : \SX \lrt \SY$ be a full, dense and objective functor and the kernel of $F$ is generated by $\SK$. Then $F$ induces an equivalence $\overline{F}: \SX/{\lan \SK \ran} \lrt \SY$. Recall that for a class $\SK$ of objects of the category $\SX$, the ideal of $\SX$ generated by all maps which factor through a direct sum of objects in $\SK$ is denoted by
$\lan \SK \ran$. Following \cite{RZ}, for the ease of notation, we just write $\SX/\SK$ instead of $\SX/{\lan \SK \ran}.$

The composition of objective functors is not necessarily objective, but if in addition, we know that they are both full and dense, then their composition is objective, full and dense.

\begin{notation}
Let $\M$ be an $n$-cluster tilting subcategory of $\mmod\La$. Let $m$ be a natural number.
\begin{itemize}
\item [$-$] We let $\ola{\M}^{\leqslant m}$  denote the subcategory of $\mmod \La$ consisting of all modules $X$ admitting a $\Hom_{\La}(\M, - )$-exact sequence \[0 \lrt M_m \lrt \cdots \lrt M_1 \lrt M_0 \lrt X \lrt 0,\] with $M_i \in \M,$ for all $i \in \{0, 1, \cdots, m\}.$ In this case, we say that $X$ has proper $\M$-dimension at most $m$. Note that since $\M$ contains projectives, the sequence itself is exact.

\item [$-$] Dually, we let $\ora{\M}^{\leqslant m}$ denote the subcategory of $\mmod \La$ consisting of all modules $X$ admitting a $\Hom_{\La}( - , \M)$-exact sequence \[0 \lrt X \lrt M^0 \lrt M^1 \lrt \cdots \lrt M^m \lrt 0,\] with $M_i \in \M,$ for all $i \in \{0, 1, \cdots, m\}.$ In this case, we say that $X$ has coproper $\M$-dimension at most $m$. Since $\M$ contains injectives, the sequence itself is exact.
\item [$-$] Let $\widetilde{\SP}^{\leqslant 1}(\M)$ (respectively, $\widetilde{\SP}^{\leqslant 1}(\M^{\op})$) denote the subcategory of $\mmod\M$ (respectively, $\M\mbox{-}{\rm mod}$) consisting of all finitely presented functors of projective dimension at most one. Note that since $\M$ is an $n$-cluster tilting subcategory of $\mmod\La$, it admits weak kernels (respectively, weak cokernels) and so $\mmod\M$ (respectively, $\M\mbox{-}{\rm mod}$) is an abelian category.
\end{itemize}
\end{notation}

Please note that in the above notations and definitions, if $m \leq n$, then using the fact that $\M$ is an $n$-cluster tilting subcategory of $\mmod\La$, we may deduce easily that any exact sequence
\[0 \lrt M_m \lrt \cdots \lrt M_1 \lrt M_0 \lrt X \lrt 0,\]
with $M_i \in \M,$ for $i \in \{0, 1, \cdots, m\},$ is automatically proper, i.e. it remains exact with respect to the functor $\Hom_{\La}(\M, - )$. Similar comment applies to the objects of $\ora{\M}^{\leqslant m}$.

\section{The functor  $\Phi: \CS(\mathcal{M}) \lrt \mmod{\underline{\mathcal{M}}}$}\label{Phi}
Let $\La$ be an artin algebra and $\M$ be an $n$-cluster tilting subcategory of $\mmod\La$, where $n>1$ is a fixed positive integer. To define $\Phi$ we need some preparations. In particular, we need to define auxiliary functors $\Up$ and $i_{\la}$. We do this in the following two subsections. Throughout the paper, the Hom functor
$\Hom_{\La}( - , M )|_{\M}$ will be denoted by $\M( - , M)$, when $M \in \M$.

\subsection{The functor $\Up: \CS(\M) \lrt \widetilde{\SP}^{\leqslant 1}(\M)$} \label{Upsilon}
In this subsection, we study a restriction to $\CS(\M)$ of the functor $\alpha$ introduced by Auslander \cite{Au1}. This functor studied further in \cite{AR2}. For more recent account see \cite[\S 3]{RZ}.

Consider the subcategory
\[\CS(\M):=\{  M_1 \st{f}\rt M_2  \ | \ f \in \CS(\La) \ {\rm{and}} \ M_1, M_2 \in \M \},\]
of $\CS(\La)$. The assignment
\[(M_1\st{f}\rt M_2) \mapsto \Coker (\M(-, M_1)\st{\M(-, f)}\lrt \M(-, M_2))\]
defines a functor
\[\Up: \CS(\M) \lrt \widetilde{\SP}^{\leqslant 1}(\M).\]
We show that this functor is full, dense and objective. To this end, we consider the following functor. Let $Y_{\M}: \mmod\La \lrt \mmod\M$ be the functor that maps $X \in \mmod\La$ to $\Hom_{\La}( - , X)| _{\M}$.

\begin{sproposition}\label{Propop2.2}
The functor $Y_{\M}$ is full and faithful. In addition, its restriction to $\ola{\M}^{\leqslant 1}$ induces an equivalence of categories
\begin{center}
$Y_{\M}|: \ola{\M}^{\leqslant 1}  \st{\simeq}\lrt \widetilde{\SP}^{\leqslant 1}(\M).$	
\end{center}
\end{sproposition}

\begin{proof}
Since $\M$ contains projective $\La$-modules, it follows that $Y_{\M}$ is full and faithful. We show that its restriction to $\ola{\M}^{\leqslant 1}$ maps to $\widetilde{\SP}^{\leqslant 1}(\M)$ and induces an equivalence. To see this, pick $X \in \ola{\M}^{\leqslant 1}$ and let $0 \lrt M \lrt M' \lrt X \lrt 0$ be an $\M$-proper resolution of $X$. It yields the exact sequence
\[ 0 \lrt \M( - , M) \lrt \M( - , M') \lrt \Hom_{\La}( - , X)|_{\M} \lrt  0.\]
This, in turn, implies that $\Hom_{\La}( - , X)|_{\M} \in \widetilde{\SP}^{\leqslant 1}(\M).$
Hence, to complete the proof, we show that $Y_{\M}|$ is dense. Let $F \in \widetilde{\SP}^{\leqslant 1}(\M)$. There exists a projective resolution
\begin{equation}\label{proj-resolution}
0 \lrt \M( - , M) \st{( - , f)} \lrt \M( - , M') \lrt F \lrt 0
\end{equation}
in $\mmod \M$, where $M, M' \in \M.$ By Yoneda's lemma, we get a monomorphism $M \st{f}{\rt} M'$. Since $\M$ is an $n$-cluster tilting subcategory of $\mmod\La$, it is $n$-abelian and hence this morphism extends to an $n$-exact sequence
\[0 \lrt M \st{f}\lrt M' \lrt N^1 \lrt \cdots \lrt N^n \lrt 0.\]
Set $C=\Coker f.$ The $n$-exactness of the sequence, induces the short exact sequence
\begin{equation}\label{proj-resolution-2}
0 \lrt \M( - , M)\st{( - , f)}\lrt \M( - , M')\lrt \Hom_{\La}( - , C)|_{\M} \lrt 0.
\end{equation}
By comparing sequences \ref{proj-resolution} and \ref{proj-resolution-2}, we get the result.
\end{proof}

\begin{sremark}\label{Exact Structure}
It is known that $\mmod\La$ can be considered as an exact category whose conflations are all $\M$-proper short exact sequences, i.e. short exact sequences of $\La$-modules that are exact under the functor $\Hom_{\La}(M , - )$, for all $M \in \M$. Note that the class of all such proper extensions corresponds to a sub-bifunctor $F_{\M}$ of the bifunctor $\Ext^1_{\La}( - , - )$, see \cite{ASo}. Since $\ola{\M}^{\leqslant 1}$ is an extension closed subcategory of $\mmod \La$ with respect to this exact structure, it also inherits an exact structure.

On the other hand, since $\widetilde{\SP}^{\leqslant 1}(\M)$  is an extension closed subcategory of the abelian category $\mmod\M$, it is an exact category with the induced exact structure: conflations are short exact sequences in $\mmod\M$ whose all terms are in  $\widetilde{\SP}^{\leqslant 1}(\M)$.

 The above proposition, in fact, provides an equivalence of exact categories $\ola{\M}^{\leqslant 1}$ and $\widetilde{\SP}^{\leqslant 1}(\M)$ via the exact functor $Y_{\M}|$.
\end{sremark}

\begin{stheorem}\label{Proposition 2.3}
The functor $\Up:\CS(\M) \lrt \widetilde{\SP}^{\leqslant 1}(\M)$ is full, dense and objective.
\end{stheorem}

\begin{proof}
The functor $\Up$ can be considered as the composition
\[\xymatrix@C=1.4cm{\Up: \CS(\M) \ar[r]^{C} & \ola{\M}^{\leqslant 1}  \ar[r]^{Y_{\M}|} & \widetilde{\SP}^{\leqslant 1}(\M), }\]
where $C$ is the usual cokernel functor.
In view of Proposition \ref{Propop2.2}, $Y_{\M}|$ is full, faithful and dense. Since faithful functors are objective, to prove the result it is enough to show that the functor
\[\xymatrix@C=1.4cm{C: \CS(\M) \ar[r] & \ola{\M}^{\leqslant 1} }\]
is full, dense and objective. We do this. To see that it is full, let $f: M_1 \rt M_2$ and $f': M'_1 \rt M'_2$ be two objects of $\CS(\M)$ and $\alpha: C( f) \lrt C( f')$ be a morphism in $\ola{\M}^{\leqslant 1}$. Consider the diagram
\[\begin{tikzcd}
0 \ar[r] & M_1  \ar[r, "f"] \ar[d, dotted, "{\alpha_1}"] & M_2 \ar[r] \ar[d, dotted, "{\alpha_2}"]  & C(f)  \ar[r] \ar[d, "{\alpha}"] & 0 \\
0 \ar[r] & M'_1 \ar[r, "{f'}"]  & M'_2 \ar[r] & C(f') \ar[r] & 0
\end{tikzcd}\]
Since $\M$ is an $n$-cluster tilting subcategory with $n>1$, rows are $\M$-proper and hence $\alpha$ can be lifted to a morphism
\[\xymatrix{f \ar[r]^{\alpha_1}_{\alpha_2} & f' }\]
in $\CS(\M)$. Hence $C$ is full. By definition, it is dense. To see that it is objective, let $(\alpha_1, \alpha_2)$ be a morphism of $f : M_1 \rt M_2$ to $f' : M'_1 \rt M'_2$ such that $C(\alpha_1, \alpha_2)=0$. Hence $(\alpha_1, \alpha_2)$ in the following diagram
\[\begin{tikzcd}
0 \ar[r] & M_1  \ar[r, "f"] \ar[d, "{\alpha_1}"] & M_2 \ar[r] \ar[d, "{\alpha_2}"]  & C(f)  \ar[r] \ar[d, "{0}"] & 0 \\
0 \ar[r] & M'_1 \ar[r, "{f'}"]  & M'_2 \ar[r] & C(f') \ar[r] & 0
\end{tikzcd}\]
is null-homotopic. Therefore, there exists a morphism $s: M_2 \rt M'_1$ such that $sf=\alpha_1$ and $f's=\alpha_2$. This, in turn, induces the following factorization of the morphism $(\alpha_1, \alpha_2)$
\[\begin{tikzcd}
0 \ar[r] & M_1  \ar[r, "f"] \ar[d, "{\alpha_1}"] & M_2 \ar[d, "{s}"] \\
0 \ar[r] & M'_1 \ar[r,  "{1}"]  \ar[d, "{1}"] & M'_1 \ar[d, "{f'}"] \\
0 \ar[r] & M'_1  \ar[r, "f'"] & M'_2
\end{tikzcd}\]
where the middle object is clearly a kernel object. This completes the proof of the proposition.
\end{proof}

We have the following immediate corollary.

\begin{scorollary}\label{equiv-Phi}
With the above notations, there exists an equivalence of additive categories
\[{\CS(\M)}/{\CK}\simeq \widetilde{\SP}^{\leqslant 1}(\M) \simeq \ola{\M}^{\leqslant 1}, \]
where $\CK$ is the full subcategory of $\CS(\M)$ generated by all isomorphisms in $\CS(\M).$
\end{scorollary}

\begin{proof}
It follows from the definition of $\Up$ that its kernel objects are isomorphisms in $\CS(\M).$ Hence, the first equivalence follows from \ref{Objective}. The second one is just Proposition \ref{Propop2.2}.
\end{proof}

\subsection{The functor $i_{\la}: \mmod\M \lrt \mmod\underline{\M}$} \label{ilambda}
Recollements for triangulated categories appeared first in \cite{BBD} to study construction of the category of perverse sheaves on a singular space. For a good account for studying recollements in abelian categories see \cite{Ps}.

Let $\La$ be an artin algebra and $\SX$ be a contravariantly finite subcategory of $\mmod \La$ containing $\prj \La$. By \cite[Theorem 3.7]{AHK}, we have a recollement
\[\xymatrix{\mmodd \SX \ar[rr]^{i}  && \mmod\SX \ar[rr]^{\va} \ar@/^1pc/[ll]^{i_{\rho}} \ar@/_1pc/[ll]_{i_{\la}} && \mmod \La, \ar@/^1pc/[ll]^{\va_{\rho}} \ar@/_1pc/[ll]_{\va_{\la}} }\]
in which $\mmodd\SX:=\Ker\va$ is the full subcategory of $\mmod\SX$ consisting of all functors $F$ such that $\va(F)=0$, equivalently, $\mmodd\CX$ consists of all functors $F \in \mmod\SX$ that vanish on $\prj \La$, the full subcategory of $\mmod\La$ of all projective $\La$-modules.

The canonical functor $\SX \lrt \underline{\SX}$ induces the functor $\mathfrak{F}: \Mod\underline{\SX} \lrt \Mod\SX$. By \cite[Proposition 4.1]{AHK} the restriction of this functor to $\mmod\underline{\SX}$ induces an equivalence of categories $\mmod\underline{\SX} \simeq \mmodd\SX$. Moreover, $\mmod\underline{\SX}$ is an abelian category. Throughout we treat $\mmod\underline{\SX}$ in this way. The above recollement will be denoted by $\SR(\SX,\La)$.

In this subsection, we are interested in the case where $\SX=\M$ is an $n$-cluster tilting subcategory of $\mmod\La$. In particular, we study the functor $i_{\la}$ explicitly. To do this, note that in view of  \cite[Proposition 2.8]{PV}, for every $F \in \mmod\M$, there exists an exact sequence
\[0 \lrt \Ker(\eta_F) \lrt \va_{\la}\va(F) \st{\eta_F} \lrt F  \lrt ii_{\la}(F) \lrt  0,\]
where $\eta_F$ denotes the counit of adjunction. So to know $i_{\la}(F)$, it is enough to know $\eta_F$. Let us first recall the definitions of the functors $\va$ and $\va_{\la}$.

Let $F \in \mmod \M$ and $\M(-, M^1) \lrt \M(-, M^0) \lrt F \lrt 0$ be a minimal projective presentation of $F.$ Then, by \cite[Proposition 3.1]{AHK}, $\va(F)$ is defined as the cokernel of the morphism $M^1 \lrt M^0$ in $\mmod \La.$ Let $\varphi: F \lrt F'$ be a morphism in $\mmod \M$. Then it can be lifted to the projective presentations of $F$ and $F'$ and hence, by applying Yoneda's lemma, we get a morphism $\va(\varphi): \va(F) \lrt \va(F')$.

Now let $M \in \mmod\La$ be an arbitrary module. Let $P\rt Q\rt M \rt 0$ be a minimal projective presentation of $M$ in $\mmod \La.$  By
\cite[Proposition 3.6]{AHK}, we set \[\va_{\la}(M) := \Coker (\M(-, P) \lrt \M(-, Q) ).\]
$\va_{\la}$ on morphisms is defined in an obvious way.

Finally, the counit $\eta_F: \va_{\la}\va(F) \lrt F$ is defined as follows. Let $P \lrt Q \lrt \va_{\la}(F) \lrt 0$ be a minimal projective presentation of $\va(F)$ in $\mmod \La.$  So we have the commutative diagram
\[\xymatrix{P \ar[r] \ar@{.>}[d] & Q \ar[r] \ar@{.>}[d] & \va(F) \ar[r] \ar@{=}[d]  & 0 \\	M^1 \ar[r] & M^0 \ar[r] & \va(F) \ar[r] & 0.}\]
 in $\mmod \La$.
Applying the Yoneda functor to the left square of this diagram leads to the following commutative diagram in $\mmod \M$
\[\xymatrix{\M( - , P) \ar[r] \ar[d] & \M( - , Q) \ar[r] \ar[d] & \va_{\la}(\va(F)) \ar[r] \ar@{.>}[d]^{\eta_F}  & 0 \\ \M( - , M^1) \ar[r] & \M( - , M^0) \ar[r] & F \ar[r] & 0.}\]
So we get the counit.

\subsection{The functor $\Phi: \CS(\M) \lrt \mmod\underline{\M}$} \label{Phi-Sub}
Now we have the necessary ingredients to introduce the functor $\Phi$. In this subsection, we assume that $\La$ is a self-injective artin algebra.

Using the notations of the previous subsections, consider the composition
\[ \CS(\M)  \st{\Up}\lrt \widetilde{\SP}^{\leq 1}(\M) \st{i_{\la}|}{\lrt}\mmod\underline{\M}\]
and set $\Phi := i_{\la}|\circ \Up.$

\begin{stheorem}\label{Phi-Theorem}
The functor $\Phi: \CS(\M)\lrt \mmod\underline{\M}$ is full, dense and objective.
\end{stheorem}

\begin{proof}
We start the proof by showing that $\Phi$ is full. Since we have already seen that $\Up$ is full, we just need to show that $i_{\la}|$ is full.
Let $F$ and $G$ be functors in $\widetilde{\SP}^{\leqslant 1}(\M)$ and $\gamma: i_{\la}(F)\rt i_{\la}(G)$ be a morphism in $\mmod \underline{\M}$.
Consider the diagram
\[\xymatrix{\va_{\la}\va(F) \ar[r]^{ \ \ \ \eta_F}  & F \ar[r] \ar@{.>}[d]^{\delta} & ii_{\la}(F) \ar[r] \ar[d]^{i\gamma}  & 0 \\ \va_{\la}\va(G) \ar[r]^{ \ \ \ \eta_G} & G \ar[r] & ii_{\la}(G)\ar[r] & 0.}\]
with exact rows. Let $K$ denotes the kernel of $G \lrt ii_{\la}(G)$. We show that $\Ext^1_{\M}(F, K)=0.$ This implies the fullness of $i_{\la}|$, as in this case we deduce that there exists a morphism $\delta: F \lrt G$ such that the right square of the above diagram is commutative, i.e. $i_{\la}|(\delta)=\gamma$.
To show the vanishing of $\Ext^1_{\M}(F, K)$ we apply the known isomorphism
\[\Ext^1_{\M}(F, K)\simeq \Hom_{\mathbb{K}}(\mathbf{P}_F, \mathbf{P}_K[1]),\]
where $\mathbb{K}$ denotes the homotopy category of complexes of functors of $\mmod \M$	and $\mathbf{P}_F$ and $\mathbf{P}_K$ denote deleted projective resolutions of $F$ and $K$, respectively. Since $F \in \widetilde{\SP}^{\leqslant 1}(\M)$, a projective resolution of it is of the form
\[0 \lrt \M( - , A)\st{\M( - , f)}\lrt \M( - , B)\rt F \rt 0,\]
where $f:A \lrt B$ is a monomorphism. Moreover, by construction of $\va_{\la}$, there is a projective presentation  $\M( - , P)\rt \M( - , Q)\rt \va_{\la}\va(F)\rt 0$ of $\va_{\la}\va(F)$ such that $P, Q \in \prj\La$. Because of the epimorphism $\va_{\la}\va(G)\rt K \rt 0$, we can choose a deleted  projective resolution of $K$
\[\mathbf{P}_G: \cdots \rt \M( - , M)\rt \M( - , N)\rt \M( - , Q)\rt 0,\]
where $\M( - , Q)$ is its zero's term.
Now consider a chain map
\[\xymatrix{\cdots \ar[r]&0 \ar[r] \ar[d]&0\ar[r] \ar[d] & \M( - , A) \ar[r] \ar[d]^{\M( - , g)} & \M(-, B) \ar[r] \ar[d]  & 0 \\ 	\cdots\ar[r] & \M( - , M) \ar[r]&	\M( - , N) \ar[r] & \M( - , Q) \ar[r] & 0 \ar[r] & 0.}\]
from $\mathbf{P}_F$ to $\mathbf{P}_G$. Since $Q$ is a projective module and $\La$ is a self-injective algebra, $Q$ is injective and hence, in view of Yoneda's lemma, the morphism $g: A \lrt Q$ can be extended to a morphism $h: B \lrt Q$. Another use of Yoneda's lemma, implies the extension of the morphism $\M( - , g)$ to a morphism from $\M( - , B)$ to $\M( - , Q)$. This, in turn, implies that the chain map is null-homotopic. Hence $\Hom_{\mathbb{K}}(\mathbf{P}_F, \mathbf{P}_K[1])=0$, so the result.

Now we show that $\Phi$ is dense. Pick $F  \in \mmod\underline{\M}$. Hence $F \in \mmod\M$ and $F|_{\prj\La}=0$. Consider a projective presentation \[\underline{\M}( - , M_1)\st{\M(-, f)}\lrt\underline{\M}(-, M_2)\]
of $F$. Let $i: M_1\lrt I$ be the injective envelop of $M_1.$ Consider the object $[f~~i]^{\rm{t}}:M_1 \rt M_2\oplus I$ in $\CS(\M).$ We claim that $\Phi([f~~i]^{\rm{t}})=F$. To see this, set
\[G:=\Up([f~~i]^{\rm{t}})=\Coker(\M( - , M_1) \lrt \M( - , M_2 \oplus I)).\]
Hence $\va(\Up([f~~i]^{\rm{t}}))=\Coker(M_1 \lrt M_2\oplus I). $ Let $Q \lrt M_2$ be a projective cover of $M_2$. Hence we get an exact sequence
\[P \lrt Q\oplus I \lrt \va(G) \lrt 0,\]
for some projective module $P$. Therefore, we get the following commutative diagram
\[\xymatrix{\M( - , P) \ar[r] \ar[d] & \M( - , Q\oplus I) \ar[r] \ar[d] & \va_{\la}\va(G)\ar[r] \ar[d] & 0 \\  \M( - , M_1) \ar[r] \ar[d] & \M( - , M_2\oplus I) \ar[r] \ar[d] & G \ar[r] \ar@{.>}[d] & 0 \\  \underline{\M}( - , M_1) \ar[r] & \underline{\M}( - ,M_2) \ar[r] & F \ar[r] & o. }\]
Hence $\Phi([f~~i]^{\rm{t}})=\Coker(\va_{\la}\va(G) \lrt G)=F$.

Finally, we show that $\Phi$ is objective. Since by Theorem \ref{Proposition 2.3}, $\Up$ is full, dense and objective, by the last paragraph of the Subsection \ref{Objective}, we just need to show that the restricted functor $i_{\la}|_{\widetilde{\SP}^{\leqslant 1}(\M)}$ is objective. To this end, let $\theta: F \lrt G$ be a morphism in $\widetilde{\SP}^{\leqslant 1}(\M)$ such that $i_{\la}|_{\widetilde{\SP}^{\leqslant 1}(\M)}(\theta)=0$. Hence we have the commutative diagram
\[\xymatrix{\va_{\la}\va(F) \ar[r]  & F \ar[r] \ar[d]^{\theta} & i_{\la}(F) \ar[r] \ar[d]^{0}  & 0 \\ \va_{\la}\va(G) \ar[r] & G \ar[r] & i_{\la}(G)\ar[r] & 0.}\]
Since the lower row is exact, $\theta$ factors through the functor $\va_{\la}\va(G)$. But by the definition of a recollement, $\va_{\la}\va(G)$ is an kernel object of $i_{\la}$. Hence the proof is complete.
\end{proof}

\begin{scorollary}\label{Phi-Corollary}
With the above notations, there exists an equivalence of abelian categories
\[{\CS(\M)}/{\CU} \simeq \mmod\underline{\M},\]
where $\CU$ is the subcategory of ${\CS(\M)}$ generated by the objects of the form $(M \st{1}\lrt M)$ and $(M \st{f} \lrt P)$, where $M \in \M$ and $P \in \prj\La.$
\end{scorollary}

\begin{proof}
By the above theorem $\Phi$ is full, dense and objective. Hence we just should note that the kernel objects of $\Phi$ are exactly those in the additive closure of a subcategory generated by all monomorphisms as in the statement. This follows easily from definition of $\Phi$. So we are done.
\end{proof}

\section{The functor $\Psi: \CS(\M) \lrt \mmod\underline{\M}$} \label{Psi}
In this section we introduce another functor on $\CS(\M)$. Throughout assume that $\La$ is an arbitrary artin algebra and $\M$ is an $n$-cluster tilting subcategory of $\mmod\La$.

Let $M_1 \stackrel{f}\rt M_2$ be an object of $\CS(\M)$. Since $\M$ is an $n$-cluster tilting subcategory, we may take an $n$-cokernel of $f$ which results to an $n$-exact sequence
\[0 \lrt M_1 \st{f}{\lrt} M_2 \st{d^1}{\lrt} M^1 \st{d^2}{\lrt} M^2  \lrt \cdots \lrt M^{n-1} \st{d^{n}}{\lrt} M^n \lrt 0 \]
Hence the following induced sequence
\[0 \rt \M( - , M_1) \rt \M( - , M_2) \rt \M( - , M^1) \rt \cdots \rt \M( - , M^n) \rt F \rt 0 \]
is exact, where $F$ is the cokernel of the morphism $\M( - , M^{n-1}) \lrt \M( - , M^n)$. Clearly $F$ vanishes on projective modules and so $F \in \mmod\underline{\M}$.
We define a functor \[\Psi: \CS(\M) \lrt \mmod\underline{\M} \] by setting $\Psi(M_1 \stackrel{f}\rt M_2)=F.$ First of all, since every two $n$-cokernels of $f$ are homotopy equivalent, we deduce that the definition of $\Psi$ is independent of the choice of the $n$-cokernel of $F$. Now let $f: M_1 \rt M_2$ and $f': M'_1 \rt M'_2$ be two objects of $\CS(\M)$ and consider a morphism $\xymatrix{f \ar[r]^{\alpha_1}_{\alpha_2} & f' }$. By the property of $n$-exact sequences, we deduce that $({\alpha_1}, {\alpha_2})$ lifts to the following morphism of $n$-exact sequences
\begin{equation}\label{equiation 2.1}
\xymatrix{0 \ar[r] & M_1 \ar[r]^{f} \ar[d]^{\alpha_1} & M_2 \ar[r]^{d^{1}} \ar[d]^{\alpha_2} & M^1 \ar[r] \ar@{.>}[d]^{\alpha^1} & \cdots \ar[r] & M^{n-1} \ar[r]^{d^n} \ar@{.>}[d]^{\alpha^{n-1}} &  M^n \ar[r] \ar@{.>}[d]^{\alpha^{n}} & 0\\ 0 \ar[r] & M'_1 \ar[r]^{f'}  & M'_2 \ar[r]^{d'^{1}} & M'^1 \ar[r] & \cdots \ar[r] & M'^{n-1} \ar[r]^{d'^n} &  M'^n \ar[r] & 0.}
\end{equation}
Yoneda's lemma now come to play to induce the commutative diagram
\[\xymatrix{0 \ar[r] & \M( - , M_1) \ar[r] \ar[d]^{\M( - , \alpha_1)} & \M(- , M_2) \ar[r] \ar[d]^{\M( - , \alpha_2)}  & \cdots \ar[r] & \M( - , M^n) \ar[r] \ar[d]^{\M( - , \alpha^{n})} & F \ar[r] \ar@{.>}[d]^{\eta} & 0\\ 0 \ar[r] & \M( - , M'_1) \ar[r]  & \M(- , M'_2) \ar[r]  & \cdots \ar[r] & \M( - , M'^n) \ar[r] & F' \ar[r] & 0.}\]
We set $\Psi({\alpha_1}, {\alpha_2})=\eta$. Comparison Lemma \cite[Lemma 2.1]{J} implies that $\eta$ is independent of the lifting morphism $\{\alpha^i\}_{1 \leq i \leq n}$.

\begin{theorem}\label{Th-Psi}
The functor $\Psi:\CS(\M) \lrt \mmod\underline{\M}$ is full, dense and objective.
\end{theorem}

\begin{proof}
Let $f: M_1 \rt M_2$ and $f': M'_1 \rt M'_2$ be two objects of $\CS(\M)$ with $\Psi(f)=F$ and $\Psi(f')=F'$. Let $\eta: F \lrt F'$ be a morphism in $\mmod\underline{\M}$. So we have the following commutative diagram
\begin{equation}\label{diag-Lifting}
\xymatrix@C=0.45cm{0 \ar[r] & \M( - , M_1) \ar[r] \ar@{.>}[d]^{\M( - , \alpha_1)} & \M(- , M_2) \ar[r] \ar@{.>}[d]^{\M( - , \alpha_2)}  & \M(- , M^1) \ar[r] \ar@{.>}[d]^{\M( - , \alpha^1)}  & \cdots \ar[r] & \M( - , M^n) \ar[r] \ar@{.>}[d]^{\M( - , \alpha^{n})} & F \ar[r] \ar[d]^{\eta} & 0\\ 0 \ar[r] & \M( - , M'_1) \ar[r]  & \M(- , M'_2) \ar[r] & \M(- , M'^1) \ar[r]  & \cdots \ar[r] & \M( - , M'^n) \ar[r] & F' \ar[r] & 0.}
\end{equation}
Since rows are projective resolutions, $\eta$ lifts to a morphism of the resolutions and hence, by Yoneda's lemma, we conclude that $\Psi$ is full.

To see that $\Psi$ is dense, pick $F \in \mmod\underline{\M}$. So there exists an exact sequence
\[\M( - , M^{n-1}) \lrt \M( - , M^n) \lrt F \lrt 0\]
such that $M^{n-1} \lrt M^n$ is an epimorphism. By taking an $n$-cokernel of this morphism, we get a monomorphism $M_1 \st{f}{\rt} M_2$ with $\Psi(f)=F$.

So it remains to prove that $\Psi$ is objective. Let $\xymatrix{f \ar[r]^{\alpha_1}_{\alpha_2} & f' }$ be a morphism of $f: M_1 \rt M_2$ to $f': M'_1 \rt M'_2$ in $\CS(\M)$ such that $\eta=\Psi({\alpha_1}, {\alpha_2})=0$. Then the lifting of $\eta$ as in diagram \ref{diag-Lifting} above is null-homotopic.
In particular, by applying Yoneda's lemma, we get morphisms $s^0: M_2 \lrt M'_1$ and $s^1:M^1 \lrt M'_2$
\[\xymatrix{0 \ar[r] & M_1 \ar[r]^{f} \ar[d]_{\alpha_1} & M_2 \ar[r]^{d^{1}} \ar[d]_{\alpha_2} \ar@{.>}[dl]_{s^0} & M^1  \ar[d]^{\alpha^1} \ar@{.>}[dl]_{s^1} \\ 0 \ar[r] & M'_1 \ar[r]^{f'}  & M'_2 \ar[r]^{d'^{1}} & M'^1 }\]
such that $\alpha_1=s^0f$ and $\alpha_2=f's^0+s^1d^1$. Therefore $({\alpha_1}, {\alpha_2})$ factors through $M'_1 \st{(1,0)}\lrt M'_1\oplus M'_2$ via the following maps:
\[\begin{tikzcd}
0 \ar[r] & M_1  \ar[r, "f"] \ar[d, "{\alpha_1}"] & M_2 \ar[d, "{{[s^0 \ s^1d^1]}^t}"] \\
0 \ar[r] & M'_1 \ar[r, "{{[1 \ 0]}^t}"]  \ar[d, "{1}"] & M'_1\oplus M'_2 \ar[d, "{[f' \ 1]}"] \\
0 \ar[r] & M'_1  \ar[r, "f'"] & M'_2.
\end{tikzcd}\]
Since the morphism $(1,0)$ in the middle row is a split monomorphism, by \cite[Proposition 2.6]{J}, the middle object is a kernel object. The proof is hence complete.
\end{proof}

\begin{corollary}\label{equiv-Psi}
With the above notations, there exists an equivalence of additive categories
\[{\CS(\M)}/{\CV}  \simeq \mmod\underline{\M},\]
where $\CV$ is the full subcategory of $\CS(\M)$ generated by all finite direct sums of objects of the form $(M\st{1}\lrt M)$ and $(0\lrt M)$, where  $M$ runs over objects of $\M$.
\end{corollary}

\begin{proof}
In view of \ref{Objective} and the above proposition, we just should show that $\CV$ is generated by the kernel objects of $\Psi$. Let $M_1 \st{f}{\rt} M_2$ be a kernel object of $\Psi$, i.e. $\Psi(f)=0$. We show that $f$ is a split monomorphism. By definition, $f$ extends to an $n$-exact sequence
\begin{equation}\label{Split-1}
0 \lrt M_1 \st{f}{\lrt} M_2 \st{d^1}{\lrt} M^1 \st{d^2}{\lrt} M^2  \lrt \cdots \lrt M^{n-1} \st{d^{n}}{\lrt} M^n \lrt 0
\end{equation}
which induces the exact sequence
\begin{equation}\label{Split-2}
0 \rt \M( - , M_1) \rt \M( - , M_2) \rt \M( - , M^1) \rt \cdots \rt \M( - , M^n) \rt 0.
\end{equation}
For $i \in \{1, \cdots, n-1\}$, set $K^i=\Ker(M^i {\lrt} M^{i+1})$. So we have short exact sequences
\[\varepsilon^i: 0 \rt K^i \rt M^i \rt K^{i+1} \rt 0,  \ i \in \{1, \cdots, n-1\},\]
and $\ 0 \rt M_1 \rt M_2 \rt K^1 \rt 0$, where $K^n:=M^n$. Apply the exact sequence \ref{Split-2} on $K^n=M^n$, implies that $\varepsilon^{n-1}$ is split. Therefore, we get exact sequence
\[ 0 \rt \M( - , M_1) \rt \M( - , M_2) \rt  \cdots \rt \M( - , M^{n-2}) \rt \M( - , K^{n-1}) \rt 0. \]
Now apply this later sequence to $K^{n-1}$ and follow this argument step by step to deduce that the short exact sequence $\ 0 \rt M_1 \st{f}\rt M_2 \rt K^1 \rt 0$ is split. Hence $f$ is a split monomorphism, as it was claimed.
\end{proof}

Here we provide two interesting applications of the equivalence of the Corollary \ref{equiv-Psi}.  Recall that an additive category $\CA$ is called of finite type if the set of all iso-classes of indecomposable objects of $\CX$ is finite. If $\mmod\La$ is of finite type, where $\La$ is an artin algebra, then $\La$ is called of finite representation type.

Let $n \geq 2$. By \cite[Lemma 2.3]{D}, if $\M$ is an $n$-cluster tilting subcategory of an exact Krull-Schmidt, Frobenius $k$-category $\SB$, then $\mmod\underline{\M}$ is of finite type if $\underline{\SB}$ is so. Our first application deals with the finiteness type of $\mmod\underline{\M}$. Note that here we do not need $\La$ to be self-injective.

\begin{proposition}
Let $n \geq 2$. Let $M \in \mmod\La$ be an $n$-cluster tilting module. Then $\Gamma=\underline{\End}_{\La}(M)$ is of finite representation type provided $\La$ is so.
\end{proposition}

\begin{proof}
For ease of notation, set $\CM=\add M$. Then $\mmod\underline{\CM} \simeq \mmod\underline{\End}_{\La}(M)$. We show that $\mmod\underline{\CM}$ is of finite type. By Corollary \ref{equiv-Psi}, we observe that $\mmod \underline{\CM}$ is of finite representation type if and only if so is $\CS(\CM)$.

Now since $\La$ is  of finite representation type, obviously every subcategory of $\mmod \La$ is of finite type. In particular,  $\overleftarrow{\CM}^{\leqslant n}$ is of finite type. Therefore we deduce from Corollary \ref{equiv-Phi} that $\CS(\CM)$ is of finite representation type.
Hence we get the result.
\end{proof}

As it is mentioned in Remark  \ref{Exact Structure},  $\mmod \La$ admits an exact structure whose conflations are all proper $\M$-exact sequences. Let $F_{\M}$ denotes the associated sub-bifunctor of the bifunctor $\Ext^1_{\La}( - , - )$. Let  ${\SP}(F_{\M})$ denote the subcategory of all relative projective modules in $\mmod\La$ with respect to this exact structure. Moreover, we denote by $\SP(\M)$ the full subcategory of $\mmod\M$ consisting of all of its projective objects.

\begin{proposition}\label{Proposition 2.10}
Let $n \geq 2$. Let $\M$ be an $n$-cluster tilting subcategory of $\mmod \La$. Then, there exist equivalences of categories
\[\mmod\underline{\M} \simeq \underline{\widetilde{\SP}^{\leqslant 1}(\M)} \simeq \underline{\ola{\M}^{\leqslant 1}}_{F_{\M}},\]
where $\underline{\widetilde{\SP}^{\leqslant 1}(\M)}$ is the stable category of the subcategory $\widetilde{\SP}^{\leqslant 1}(\M)$ of $\mmod \M$ in the usual sense, and $\underline{\ola{\M}^{\leqslant 1}}_{F_{\M}}$ is the stable category of the subcategory  $\ola{\M}^{\leqslant 1}$ of $\mmod \La$ with respect to the exact structure induced by the sub-bifunctor $F_{\M}$ of $\Ext^1_{\La}( - , - )$.	
\end{proposition}

\begin{proof}
By Theorem \ref{Proposition 2.3}, the functor $\Up: \CS(\M) \lrt \widetilde{\SP}^{\leqslant 1}(\M)$ is full, dense and objective. Moreover, it sends objects of $\CV$ to projective objects of $\mmod \M$. Hence induces an equivalence
\[\CS(\M)/\CV \simeq \underline{\widetilde{\SP}^{\leqslant 1}(\M)}.\]
So the first equivalence follows from the Corollary \ref{equiv-Psi}. The second equivalence follows from Proposition \ref{Propop2.2} in view of the fact that the functor $Y_{\M}|$ also provides an equivalence between $\SP(\M)$ and ${\SP}(F_{\M}).$
\end{proof}

We have the following immediate corollary for $n=2$.

\begin{corollary}\label{corollary-2}
Let $\M$ be a $2$-cluster tilting subcategory of $\mmod\La$. Then there is the following equivalences of categories
\[\mmod\underline{\M} \simeq \underline{\widetilde{\SP}^{\leqslant 1}(\M)} \simeq \underline{\mmod \La}_{F_{\M}}.\]
\end{corollary}

\begin{proof}
Let $X \in \mmod\La$ and consider a right $\M$-approximation $M \st{\alpha}\twoheadrightarrow X$ of $X$. Since $\M$ is a $2$-cluster tilting subcategory of $\mmod\La$, $\Ext^1_{\La}(\M, \M)=0$. This implies that $\Ext^1_{\La}(\M, \Ker\alpha)=0.$ Hence $\Ker\alpha \in \M$ and so $X \in \ola{\M}^{\leqslant 1}$. Therefore  $\mmod\La=\ola{\M}^{\leqslant 1}$. Now the result follows from the above proposition.
\end{proof}

\section{Comparison}\label{Comp}
In this section we compare the functors $\Psi$ and $\Phi$. Such comparison is inspired by \cite[Theorem 2]{RZ} and \cite[Theorem 4.2]{E}. Throughout assume that $\La$ is a self-injective artin algebra.

Let $\Omega_{\underline{\M}}:\unmod\underline{\M} \lrt \unmod\underline{\M}$ denote the syzygy functor. Note that since $\mmod \underline{\M}$ is semi-perfect, we can assume that $\Omega_{\underline{\M}}(F)$ is the kernel of a projective cover of $F$ in $\mmod\underline{\M}$. Let $\CW$ denote the smallest additive subcategory of $\CS(\M)$ generated by all objects of the form $M\st{1}\lrt M, \ 0 \lrt M$ and $M\lrt P$, with $P \in \prj\La$.

\s There exists an induced functor $\CS(\M)/\CW \lrt \unmod\underline{\M}$ that will be denoted by $\underline{\Phi}$. To see this, note that by definition
$\Phi(M \st{1_M}\lrt M)=0$. Set $F=\Coker(\M( - , M) \lrt \M( - , P))$. So $\va(F)=\Coker(M \lrt P)$. Consider a projective presentation of $\va(F)$ with $P$ as the zero's term to get the following commutative diagram
\[\xymatrix{ ( - , Q) \ar[r] \ar[d] & ( - , P) \ar[r] \ar@{=}[d] & \va_{\la}\va(F) \ar[r] \ar[d] & 0 \\ ( - , M) \ar[r] & ( - , P) \ar[r] & F \ar[r] & 0 }\]
Now Snake lemma implies that $\Phi(M \lrt P) =\Coker(\va_{\la}\va(F) \lrt F) = 0.$
Finally, we show that $\Phi(0 \lrt M)=\underline{\M}( - , M)$. Clearly $\Up(0 \lrt M)=\M( - , M)$ and $\va( - , M)=M$. So if we let $Q \lrt P \lrt M \lrt 0$ be a projective presentation of $M$, the claim follows from the following commutative diagram
\[\xymatrix{\M( - , Q) \ar[r] \ar[d] & \M( - , P) \ar[r] \ar[d] & \va_{\la}\va( - , M) \ar[r] \ar[d] & 0 \\ 0 \ar[r] \ar[d] & \M( - , M) \ar[r] \ar[d] & \M( - , M) \ar[r] \ar[d] & 0\\
0 \ar[r] \ar[d] & \underline{\M}( - , M) \ar[r]^{\simeq} \ar[d] & i_{\la}( - , M) \ar[r] \ar[d] & 0\\ 0 & 0 & 0}\]

\s The notion of $n\Z$-cluster tilting subcategories is introduced by Iyama and Jasso \cite{IJ}, as subcategories that are closed under $n$-syzygies and $n$-cosyzygies and so are ``better behaved from the viewpoint of higher homological algebra''. Recall that an $n$-abelian category $\M$ has $n$-cosyzygies if for every $M \in \M$ there exists an $n$-exact sequence
\[0\lrt  M \lrt  I^1 \lrt \cdots \lrt I^n \lrt L \lrt 0\]
such that $I^i$ is injective, for $i \in \{1,  \cdots, n\}$, and $L \in \M$ \cite[Definition 2.22]{IJ}. By abuse of notation, we say that $L$ is the $n$-cosyzygy of $M$ and denote it by $\Om^{-n}_{\La}M$. The notion of $n$-syzygies is defined dually. We denote the $n$-syzygy of $M$ by $\Om^{-n}_{\La}M$.

We say that an $n$-cluster tilting subcategory $\M$ of $\mmod\La$ is $n\Z$-cluster tilting if it admits $n$-syzygies, i.e. $\Om^n(\M) \subseteq \M$ or equivalently, if it admits $n$-cosyzygies, i.e. $\Om^{-n}(\M) \subset \M.$ See \cite[Definition-Proposition  2.15]{IJ} for more equivalent statements.

Now assume that $\M$ is an $n\Z$-cluster tilting subcategory. Then, by definition, $\Psi(0 \lrt M) = 0$, $\Psi(M \lrt M)=0$ and, using the fact that $\La$ is self-injective, $\Psi(M \lrt P)= \underline{\M}( - , \Omega^{-n}_{\La}(M))$. Hence in this case also we have the induced functor
$\underline{\Psi}: \CS(\M)/\CW \lrt \unmod\underline{\M}$.

Now we are in a position to prove the main result of this subsection.

\begin{theorem}\label{Th-Comparison}
Let $\La$ be a self-injective artin algebra and $\M$ be an $n\Z$-cluster tilting subcategory of $\mmod \La.$ Then, with the above notations, we have
\[{\underline{\Phi}} = {\Omega^n_{\underline{\M}}} \circ {\underline{\Psi}}.\]
That is, the functors ${\underline{\Phi}}$ and ${\underline{\Psi}}$ differ by the $n$-syzygy functor on $\unmod\underline{\M}$.
\end{theorem}

\begin{proof}
For every $X, Y \in \mmod \La$ and each $i>0$, there exists a functorial isomorphism
\[\underline{\Hom}_{\La}(\Omega^i(X), Y) \simeq \Ext^i_{\La}(X, Y).\]
Using these isomorphisms, one can show that $\underline{\M}$ is an $n$-cluster tilting subcategory of the triangulated category $\underline{\rm{mod}}\mbox{-}\La$. On the other hand, since $\M$ is an $n\Z$-cluster tilting subcategory, it is closed under $n$-cosyzyies and hence by \cite[Theorem 1]{GKO} its stable category $\underline{\M}$ is an $(n+2)$-angulated category. Moreover, any $n$-exact sequence in $\M$ induces an $(n+2)$-angle in $\underline{\M}$.
Let $(M_1 \st{f}\rt M_2)$ be an object in $\CS(\M)$. It completes to an $n$-exact sequence
\[0 \lrt M_1 \st{f}{\lrt} M_2 \st{d^1}{\lrt} M^1 \lrt \cdots M^{n-1} \st{d^{n}}{\lrt} M^n \lrt 0.\]
So, by the structure of the $n$-angulated categories, it induces the following $(n+2)$-angle
\[\underline{M}_1 \st{\underline{f}}{\lrt} \underline{M}_2 \st{\underline{d}^1}{\lrt} \underline{M}^1 \lrt \cdots \underline{M}^{n-1} \st{\underline{d}^n}{\lrt} \underline{M}^n \lrt \Omega^{-n}_{\La}(\underline{M}_1)\]
in $\underline{\M}$. Thanks to \cite[Proposition 2.5]{GKO}, the above $(n+2)$-angle induces a long exact sequence of functors in $\mmod\underline{\M}$
\[\cdots \rt ( - , \Omega^n_{\La}(\underline{M}^n)) \rt ( - , \underline{M}^1) \rt \cdots \rt ( - , \underline{M}^n) \rt ( - , \Omega^{-n}_{\La}(\underline{M}_1)) \rt \cdots\]
Note that, for brevity, in the above sequence we have used $( - , \underline{M})$ instead of $\underline{\M}( - , \underline{M})$. Set $\underline{F}=\underline{\Psi}(M_1 \st{f}\rt M_2)$ and $\underline{G}=\underline{\Phi}(M_1 \st{f}\rt M_2)$. By definitions of $\Psi$ and $\Phi$, one can see that $\underline{F}$ is just the cokernel of the morphism
$\underline{\M}( - , \underline{M}^{n-1}) \lrt \underline{\M}( - , \underline{M}^n)$ appeared in the above long exact sequence while $G$ is the cokernel of the morphism  $\underline{\M}( - , \underline{M}_1) \st{\underline{\M}( - , \underline{f})} \lrt \underline{\M}( - , \underline{M}_1)$ appeared in the same exact sequence. Hence, in view of this long exact sequence, we get the exact sequence
\[0 \rt \underline{G} \rt \underline{\M}( - , \underline{M}^1) \rt \cdots \underline{\M}( - , \underline{M}^{n-1}) \rt \underline{\M}( - , \underline{M}^n) \rt \underline{F} \rt 0.\]
 in $\mmod \underline{\M}$, that completes the proof.	
\end{proof}

\section{Dual statements on $\CF(\M)$}\label{Sec-Dual}
In order to provide some applications, we need to have dual of the results we had so far. Since almost all of the proofs are similar, or rather dual, we just summarize the dual statements without proof.

\subsection{The functor $\Phi': \CF(\M) \lrt \overline{\M}\mbox{-}{\rm{mod}} $} \label{Phi'}
To define $\Phi'$, similar to what we did in Section \ref{Phi}, we need to define functors $\Up'$ and $i'_{\la}$.

\sss {\sc The functor $\Up': \CF(\M) \lrt \widetilde{\SP}^{\leqslant 1}(\M^{\op})$.}\label{Upsilon'}

The functor $Y'_{\M}:\mmod\La \lrt \M\mbox{-}\rm{mod}$ defined by $X \mapsto \Hom_{\La}(X, - )| _{\M}$ is full and faithful. In addition, the restricted functor $Y'_{\M}|:\ora{\M}^{\leqslant 1} \lrt \widetilde{\SP}^{\leqslant 1}(\M^{\op})$ is a duality. Compare Proposition \ref{Propop2.2}.

\begin{stheorem} (Compare Theorem \ref{Proposition 2.3} and Corollary \ref{equiv-Phi}). \label{Th-Up'}
The contravariant functor $\Up':\CF(\M) \lrt \widetilde{\SP}^{\leqslant 1}(\M^{\op})$ defined by the composition
\[\xymatrix@C=1.4cm{\Up': \CF(\M) \ar[r]^{K} & \ora{\M}^{\leqslant 1}  \ar[r]^{Y'_{\M}|} & \widetilde{\SP}^{\leqslant 1}(\M^{\op}) }\]
which maps $(M_1\st{f}{\rt} M_2)$ to $\Coker (\M(M_2, - )\st{(f , - )}\lrt \M(M_1, - ))$ is full, dense and objective. In particular,  there is an equivalence of additive categories
\[{\CF(\M)}/{\CK'}\simeq {(\widetilde{\SP}^{\leqslant 1}(\M^{\op}))}^{\op}\simeq \ora{\M}^{\leqslant 1},\]
where $\CK'$ is the subcategory of $\CF(\M)$ generated by all isomorphisms in $\CF(\M)$.
\end{stheorem}

\subsection{The functor $i'_{\la}: \M\mbox{-}{\rm{mod}} \lrt \overline{\M}\mbox{-}{\rm{mod}} $} \label{ilambda'}
Let $\La$ be an artin algebra and $\M$ be an $n$-cluster tilting subcategory of $\mmod\La$. By \cite[Theorem 3.7]{AHK}, there exists a recollement
\[\xymatrix{\overline{\M}{\mbox{-}}{\rm{mod}} \ar[rr]^{i'}  && \M\mbox{-}{\rm mod} \ar[rr]^{\va'} \ar@/^1pc/[ll]^{i'_{\rho}} \ar@/_1pc/[ll]_{i'_{\la}} && (\mmod \La)^{\op} \ar@/^1pc/[ll]^{\va'_{\rho}} \ar@/_1pc/[ll]_{\va'_{\la}} }\]
of abelian categories, where $\overline{\M}{\mbox{-}}{\rm{mod}}=\Ker\va'$ is the full subcategory of $\CX{\mbox{-}\rm{mod}}$ consisting of all functors that vanish on injective modules.

Therefore similar to Subsection \ref{ilambda} one can define $i'_{\la}$ explicitly and then define the contravariant functor $\Phi'$ as the following composition
\[\Phi':  \CF(\M) \st{\Up'}\lrt  \widetilde{\SP}^{\leqslant 1}(\M^{\op}) \st{i'_{\la}|}{\lrt} \overline{\M}\mbox{-}{\rm{mod}}.\]
That is $\Phi' := i'_{\la}|\circ \Up'.$

\begin{stheorem} (Compare Theorem \ref{Phi-Theorem} and Corollary \ref{Phi-Corollary}). \label{Phi'-Theorem}
Let $\La$ be a self-injective artin algebra. The contravariant functor $\Phi': \CF(\M) \lrt \overline{\M}\mbox{-}{\rm{mod}}$ is full, dense and objective. In particular, there exists an equivalence of abelian categories
\[{\CF(\M)}/{\CU'} \simeq {(\overline{\M}\mbox{-}{\rm{mod}})}^{\op},\]
where $\CU'$ is the subcategory of ${\CF(\M)}$ generated by all objects of the form $(M \st{1}\lrt M)$ and $(I \st{f} \lrt M)$, where $M \in \M$ and $I \in \inj\La.$
\end{stheorem}

\subsection{The functor $\Psi':  \CF(\M) \lrt \overline{\M}\mbox{-}{\rm{mod}}$} \label{Psi'}
In analog with the definition of the functor $\Psi$ at Section \ref{Psi}, we can define the contravariant functor $\Psi'$. Let us review the definition briefly.

Pick an epimorphism $M_1 \st{f}\rt M_2$ of $\CF(\M)$. By taking $n$-kernel in $\M$ we get the $n$-exact sequence
\[0 \lrt M^1\st{d^1}\lrt M^2\st{d^2}\lrt  \cdots \lrt M^n \st{d^n}\lrt M_1 \st{f}\lrt M_2 \lrt 0.\]
The $n$-exactness induces the exact sequence
\[0 \rt \M(M_2, - ) \rt \M(M_1, - ) \rt \cdots \rt \M(M^2, - ) \rt \M(M^1, - ) \rt F \rt 0.\]
Define $\Psi'(M_1 \st{f}\rt M_2):=F.$ Since the restriction of $F$ on injective modules is zero, $F$ is indeed an object of   $\overline{\M}\mbox{-}\rm{mod}$. The action of $\Psi'$ on morphisms defines naturally.

\begin{stheorem}(Compare Theorem \ref{Th-Psi} and Corollary \ref{equiv-Psi})\label{Theorem Psi'}
The contravariant functor $\Psi':\CF(\M) \lrt \overline{\M}\mbox{-}\rm{mod}$ is full, dense and objective. In particular, there exists an equivalence of abelian categories
\[{\CF(\M)}/{\CV'}  \simeq {(\overline{\M}\mbox{-}\rm{mod})}^{\op},\]
where $\CV'$ is the full subcategory of $\CF(\M)$ generated by all finite direct sums of objects of the form $(M\st{1}\rt M)$ and $(M\rt 0)$, where  $M$ runs over objects of $\M$.
\end{stheorem}

As a dual of Proposition \ref{Proposition 2.10} we have the following result.

\begin{sproposition}  \label{Prop 2.17}
Let $n \geq 2$. Let $\M$ be an $n$-cluster tilting subcategory of $\mmod \La$. Then, there exist equivalences of categories
\[\overline{\M}\mbox{-}\rm{mod} \simeq (\underline{\widetilde{\SP}^{\leqslant 1}(\M^{\rm{op}})})^{\rm{op}} \simeq \underline{\ora{\M}^{\leqslant 1}}_{F^{\M}},\]	
where $\underline{\ora{\M}^{\leqslant 1}}_{F^{\M}}$ is the stable category of the subcategory  $\ora{\M}^{\leqslant 1}$ of $\mmod \La$ with respect to the exact structure induced by the sub-bifunctor $F^{\M}$ of $\Ext^1_{\La}( - , - )$.	
\end{sproposition}

We end this section by comparing the functors $\Phi'$ and $\Psi'$.

\begin{theorem} (Compare Theorem \ref{Th-Comparison}.)
Let $\La$ be a self-injective artin algebra and $\M$ be an $n\Z$-cluster tilting subcategory of $\mmod \La.$ Then, with the above notations, we have
 \[{\overline{\Phi'}} = \Omega^{-n}_{\overline{\M}} \circ {\overline{\Psi'}}.\]
That is, the functors $\overline{\Phi'}$ and $\overline{\Psi'}$ differ by the $n$-cosyzygy functor on $\overline{{\rm mod}}\mbox{-}\overline{\M}.$
\end{theorem}

\section{Applications}
In this section we plan to present some applications of our results.
Let us begin by recalling the definition of a functor that already has observed by Auslander \cite{Au3}. Define a functor $\Theta: \CF(\M) \lrt \mmod\underline{\M}$ by
\[\Theta(M_1\st{f}\rt M_2)=\Coker(\M( - , M_1) \st{\M( - ,  f)}{\lrt} \M( - , M_2)). \]
Note that since $f$ is an epimorphism, the cokernel is an object of $\mmod \underline{\M}$.

The proof of the following theorem is similar to the proof of Theorem \ref{Th-Psi}.

\begin{theorem}\label{Th-Theta}
The functor $\Theta$ is full, dense and objective. In particular, it induces an equivalence of categories
\[\CF(\M)/{\CV'} \simeq \mmod\underline{\M},\]
where $\CV'$ is the full subcategory of $\CF(\M)$ generated by all objects of the form $(M \st{1}{\rt} M)$ and $(M \rt 0)$ and $M$ runs over all objects of $\M$.
\end{theorem}

As a consequence we have the following interesting result.

\begin{corollary}\label{Corollary-Auslander}
There are equivalences of abelian categories
\[\CS(\M)/{\CV}  \simeq \mmod\underline{\M}  \simeq {(\overline{\M}\mbox{-}\rm{mod})}^{\op} \simeq \CF(\M)/{\CV'}.\]
\end{corollary}

\begin{proof}
The proof follows immediately from Corollary \ref{equiv-Psi} and Theorem \ref{Theorem Psi'} in conjunction with the above theorem.
\end{proof}

\begin{remark}\label{Remark-Sigma}
We let \[\Sigma=\Psi'\circ\Theta^{-1}: \mmod\underline{\M}  {\lrt} \overline{\M}\mbox{-}\rm{mod}\]
denote the duality between $\mmod\underline{\M}$ and $\overline{\M}\mbox{-}\rm{mod}$, introduced in Corollary \ref{Corollary-Auslander}. Note that Auslander \cite{Au3} has proved the existence of a duality between $\mmod\underline{\A}$ and $\overline{\A}\mbox{-}{\rm mod}$, where $\CA$ is an abelian category. Hence, the duality $\Sigma$ of the above corollary can be thought of as a higher version of the Auslander's duality.
\end{remark}

We plan to study the duality $\Sigma$ a little bite more. To this end, let us recall the notion of the defect of an $n$-exact sequence \cite[Definition 3.1]{JK}. Let $\M$ be a subcategory of $\mmod\La$ and
\[\delta: 0 \rt M^0 \st{f^0}\rt M^1 \rt \cdots \rt M^n\st{f^n}\rt M^{n+1}\rt 0\]
be an $n$-exact sequence in $\M$. The contravariant defect of $\delta$, denoted by $\delta^*$, is defined by the exact sequence
\[\Hom_{\La}( - , M^n) \lrt \Hom_{\La}( - , M^{n+1}) \lrt \delta^* \lrt 0,\]
of functors. Dually, the covariant defect of $\delta$, denoted by $\delta_*$, is defined by the following exact sequence of functors
\[ \Hom_{\La}(M^1, - ) \lrt \Hom_{\La}(M^0, - ) \lrt \delta_* \lrt 0.\]

We also need the following easy lemma.

\begin{lemma}\label{Lemma 2.22}
Let $\M$ be an $n$-cluster tilting subcategory of $\mmod\La$. Let
\[\delta: 0 \rt M^0 \st{f^0}\rt M^1 \rt \cdots \rt M^n\st{f^n}\rt M^{n+1}\rt 0\]
be an $n$-exact sequence in $\M$. The following assertions hold.
\begin{itemize}
\item [$(i)$] If $M^n$ belongs to $\prj\La$, then $\delta^* \cong \underline{\M}( - , M^{n+1})$ and $\delta_* \cong \Ext^n_{\La}(M^{n+1}, - )|_{\M}$.
\item [$(ii)$] If $M^1$ belongs to $\inj\La$, then $\delta^* \cong \Ext^n_{\La}( - , M^0)|_{\M}$ and $\delta_* \cong \overline{\M}(M^0, - ).$
\end{itemize}
\end{lemma}

\begin{proof}
We just prove part $(i)$. The statement $(ii)$ follows similarly. So assume that $M^n$ is a projective $\La$-module. In this case, the contravariant defect follows by definition.  For the covariant defect, set $K^i:=\Ker( f^i)$ for $i \in \{1, \cdots, n\}$. Since $\M$ is an $n$-cluster tilting subcategory of $\mmod\La$, we have $\Ext^i_{\La}(\M, \M)=0$, for $i \in \{1, \cdots, n-1\}$. This implies that
$\delta_*\cong \Ext^1_{\La}(K^2, - )|_{\M}$ and then dimension shifting argument in view of the short exact sequences
\[0 \lrt K^i \lrt M^i \lrt K^{i+1} \lrt 0, \ \ i \in \{2, \cdots, n\}, \]
applies to show that $\delta_* \cong \Ext^n_{\La}(M^{n+1}, - )|_{\M}$, where by convention $K^{n+1}=M^{n+1}$.
\end{proof}

\begin{proposition}\label{Prop-Sigma}
Let  $\Sigma: \mmod\underline{\M} \lrt \overline{\M}\mbox{-}{\rm mod}$ be the duality of Remark \ref{Remark-Sigma}, where $\M$ is an $n$-cluster tilting subcategory of $\mmod \La$. Then for every $n$-exact sequence $\delta$ of $\M$, $\Sigma(\delta^*)=\delta_*$ and $\Sigma(\delta_*)=\delta^*$. In particular, we have the following statements.
\begin{itemize}
\item [$(i)$] $\Sigma|: \prj\mmod\underline{\M} \rt \inj\overline{\M}\mbox{-}{\rm mod}$ is defined by $\Sigma|(\underline{\M}( - , X))=\Ext^n_{\La}(X, - )|_{\M}$ and is a duality.
\item [$(ii)$] $\Sigma^{-1}|:\prj\overline{\M}\mbox{-}{\rm mod} \rt \inj\mmod\underline{\M}$ is defined by $\overline{\M}(X, - ) \mapsto \Ext^n( - , X)|_{\M}$ and is a duality.
\end{itemize}
\end{proposition}

\begin{proof}
 The facts that $\Sigma(\delta^*)=\delta_*$ and $\Sigma(\delta_*)=\delta^*$ follow directly by definition of $\Sigma$. For the second part, we just prove the statement $(i)$. Statement $(ii)$ follows similarly. It is known that the projective functors of $\mmod\underline{\M}$ are just representable functors. Let $\underline{\M}( - , X)$ be a projective object. Consider the projective cover $P \rt X$ of $X$. By taking its $n$-kernel we get an $n$-exact sequence. Now statement $(i)$ of Lemma \ref{Lemma 2.22} in view of the first part of the proposition implies the result.
\end{proof}

As an immediate consequence of the above proposition, we prove a higher version of Hilton-Rees theorem for $n$-cluster tilting subcategories. For a recent account on the Hilton-Rees theorem and a `short and straightforward proof' of it see \cite[\S 4]{M}.

\begin{theorem} (Higher Hilton-Rees)\label{Th-HiltonRees}
Let $\M$ be an $n$-cluster tilting subcategory of $\mmod\La$, and $X, Y$ in $\M.$
\begin{itemize}
\item [$(i)$] There is an isomorphism between $\underline{\M}(X, Y)$ and the group of natural transformations from $\Ext^n_{\La}(X, - )|_{\M}$ to
$\Ext^n_{\La}(Y, - )|_{\M}$.
\item [$(ii)$] There is an isomorphism between $\overline{\M}(X, Y)$ and the group of natural transformations from $\Ext^n_{\La}( - , X)|_{\M}$ to $\Ext^n_{\La}( - ,Y)|_{\M}$.
\end{itemize}
\end{theorem}

\begin{proof}
Part $(i)$ folows from the duality $\Sigma|$ of part $(i)$ of the above proposition in view of Yoneda's lemma. Part $(ii)$ follows similarly.
\end{proof}

\begin{remark}
The above two results are known over an $n$-abelian category with enough projective and enough injective objects, see Proposition 4.35 and Theorem 4.36 of \cite{Li}.  So one can conclude that they also hold true for $n\Z$-cluster tilting subcategories. In fact, here we extend them to any $n$-cluster tilting subcategory of $\mmod\La$.
\end{remark}

Our next aim is to state and prove a higher version of Auslander's direct summand conjecture \cite{Au1}. The conjecture says that for an object $A$ of an abelian category $\A$ with enough projective objects, any direct summand $F$ of $\Ext^1_{\A}(A, - )$ is of the form $\Ext^1_{\A}(B, - )$, for some $B$ in $\A.$ For a review of the conjecture and related results see the introduction of \cite{M}. We just mention that, as Auslander proved \cite[Proposition  4.3]{Au1}, if the above conjecture holds true, then functors of the form $\Ext^1_{\A}(A, - )$ are the only injectives in $\mmod\underline{\A}$. A relative version of this conjecture is proved in \cite[Theorem 3.10]{H2}.

\begin{theorem}\label{Th-DirectSummand}
Let $\M$ be an $n$-cluster tilting subcategory of $\mmod \La$.
\begin{itemize}
\item [$(i)$] If $F$ is a direct summand of $\Ext^n_{\La}(A, - )|_{\M}$, then there exists $B \in \M$ such that $F \simeq \Ext^n_{\La}(B, - )|_{\M}$.
\item [$(ii)$] If $F$ is a direct summand of $\Ext^n_{\La}( - , A)|_{\M}$, then there exists $B \in \M$ such that $F \simeq \Ext^n_{\La}( - , B)|_{\M}$.
\end{itemize}	
\end{theorem}

 \begin{proof}
 $(i).$ Let $F$ be a direct summand of $\Ext^n_{\La}(A,-)|_{\M}$. By Proposition \ref{Prop-Sigma}, $\Ext^n_{\La}(A, - )|_{\M}$ is isomorphic to $\Sigma|(\underline{\M}( - , A)).$ Since $\Sigma$ is an additive functor, there exists a summand $G$ of $\underline{\M}( - , A)$ such that $\Sigma(G)=F$. But $G$, as a summand of $\underline{\M}( - , A)$ should be of the form $\underline{\M}( - , B)$, for some summand  $B$ of $A$, as $\M$ is closed under direct summands. Hence $F \cong \Ext^n_{\La}(B, - )|_{\M}$, as desired.
 \end{proof}

Next result of the paper reproves the existence of $n$-Auslander-Reiten translation $\tau_n=\tau\Omega^{n-1}_{\La}$ that is already known by \cite{I1}. Our proof provides a functorial approach for the existence of $n$-Auslander-Reiten translation.

\begin{theorem}\label{Th-nAuslanderReiten}
Let $\M$ be an $n$-cluster tilting subcategory of $\mmod \La$. Then, there is an equivalence $\tau_n: \underline{\M} \lrt \overline{\M}$ such that for every $X, Y \in \M$,
\[\Ext^n_{\La}(X, Y) \cong D\overline{\M}(Y, \tau_n(X)).\]	
\end{theorem}

\begin{proof}
Since $\M$ is a functorially finite subcategory of $\mmod\La$, $\overline{\M}$ is a $k$-dualizing variety. Hence there exists a duality $D:\overline{\M}\mbox{-}{\rm mod} \lrt \mmod\overline{\M}.$ Consider the composition
\[D \circ \Sigma: \mmod\underline{\M} \st{\Sigma}\lrt \overline{\M}\mbox{-}\rm{mod} \st{D}\lrt \mmod\overline{\M}.\]
It is obvious that the restriction of this composition  to projective objects
\[D \circ \Sigma|: \prj\mmod\underline{\M} \lrt \prj\overline{\M}\mbox{-}\rm{mod}\]
is also an equivalence. On the other hand, by Yoneda's lemma, we have the equivalences $\SY: \underline{\M} \simeq  \prj\mmod\underline{\M}$ and
$\SY': \overline{\M} \simeq \prj\overline{\M}\mbox{-}\rm{mod}$. So, altogether we get the equivalence $\tau_n: \underline{\M} \rt \overline{\M}$ as $\tau_n:=\SY'^{-1}\circ D \circ \Sigma| \circ \SY$. Following commutative diagram explains what have been done
\[\begin{tikzcd}
\mmod\underline{\M} \ar[r, "{\Sigma}"] & \overline{\M}\mbox{-}\rm{mod} \ar[r, "{D}"] & \mmod \overline{\M} \\
\prj\mmod\underline{\M} \ar[u, hook, "\ell"] \ar[rr, "{D \circ \Sigma|}"] && \prj\overline{\M}\mbox{-}\rm{mod} \ar[u, hook, "{\ell'}"]\\
\underline{\M} \ar[u, "\SY"] \ar[rr, "{\tau_n}"] && \overline{\M} \ar[u, "{\SY'}"]
\end{tikzcd}\]
Now the commutativity of the diagram, implies that
\[\Sigma(\underline{\M}( - , X)) = D(\overline{\M}( - , \tau_n(X))),\]
for each $X \in \M$. This, in view of Proposition \ref{Prop-Sigma}, implies that
\[\Ext^n_{\La}(X, - )|_{\M} \simeq D(\overline{\M}( - , \tau_n(X))),\]
which is the desired isomorphism. The proof is hence complete.
\end{proof}

We end this section with another direct application of the above results. In fact, the equivalences given in  Propositions \ref{Proposition 2.10} and \ref{Prop 2.17} and Corollary \ref{Corollary-Auslander} establish the following equivalence between the stable categories of the functors of projective dimension at most one.

\begin{proposition}\label{stableequiv}
Let $\CM$ be an $n$-cluster tilting subcategory of $\mmod \La,$ where $n\geq 2$. Then, there exists the equivalence
\[\underline{\widetilde{\mathcal{P}}^{\leqslant 1}(\CM)}\simeq \underline{\widetilde{\mathcal{P}}^{\leqslant 1}(\CM^{\rm{op}})}.\]
In particular, by using the duality $D: {\CM}\mbox{-}{\rm mod} \lrt \mmod{\CM},$ we get the following equivalence
\[\underline{\widetilde{\mathcal{P}}^{\leqslant 1}(\CM)}\simeq \overline{\widetilde{\mathcal{I}}^{\leqslant 1}(\CM)},\]
where $\widetilde{\mathcal{I}}^{\leqslant 1}(\CM)$ denotes the subcategory of $\mmod \CM$ consisting of all functors of injective dimension at most $1$.
\end{proposition}

Recall that an artin algebra $\La$ is called $n$-Auslander \cite{I2} provided $\gldim\La \leq n+1 \leq \domdim\La,$ where $\gldim\La$ denotes the global dimension of $\La$ and $\domdim\La$ denotes the dominant dimension of $\La$ introduced by Tachikawa \cite{Ta}.

\begin{corollary}
Let $n\geqslant 2$ and $\La$ be an $n$-Auslander algebra. Then, there exist the following equivalences
\[\underline{\widetilde{\mathcal{P}}^{\leqslant 1}(\La)}\simeq \underline{\widetilde{\mathcal{P}}^{\leqslant 1}(\La^{\rm{op}})} \ \text{and } \ \underline{\widetilde{\mathcal{P}}^{\leqslant 1}(\La)}\simeq (\overline{\widetilde{\mathcal{I}}^{\leqslant 1}(\La)})^{\rm{op}}.\]
\end{corollary}

 \vspace{0.5cm}


\end{document}